\documentclass[reqno]{amsart}%
\usepackage{amstext}
\usepackage{amsfonts}
\usepackage{amsmath}
\usepackage{amssymb}
\usepackage{hyperref}
\usepackage{graphicx}%
\providecommand{\U}[1]{\protect\rule{.1in}{.1in}}
\numberwithin{equation}{section}
\newtheorem{theorem}{Theorem}[section]

\newtheorem{proposition}[theorem]{Proposition}

\begin{document}
\title[Regularity for the 3D Navier-Stokes equations ]{Global regularity criterion for the 3D Navier-Stokes equations with large data}
\author{Abdelhafid Younsi}
\address{Department of Mathematics and Computer Science, University of Djelfa , Algeria.}
\email{younsihafid@gmail.com}

\begin{abstract}
In this paper, we study the global regularity of strong solution to the Cauchy
problem of 3D incompressible Navier-Stokes equations with large data and
non-zero force. We prove that the strong solution exists globally for $\nabla
u\in L^{\alpha}\left(  0,T,L^{2}\left(  \Omega\right)  \right)  $ with
$2\leq\alpha\leq4$.

\end{abstract}
\maketitle

\section{Introduction}

Leray (1934) \cite{8} and Hopf (1951) \cite{6} showed the existence of global
weak solutions to the three-dimensional Navier-Stokes system, but their
uniqueness is still a fundamental open problem \cite{4} and \cite{5}.
Furthermore, the strong solutions for the 3D Navier-Stokes equations are
unique and can be shown to exist on a certain finite time interval for small
initial data and small forcing term (see \cite{2, 4, 5, 7} and references
therein). The global existence of small strong solutions has been proved, see
Constantin \cite{4}. For the 3D Navier-Stokes equations with large data, we
don't have a result of global existence and the question of asymptotic
behavior\ of strong solution, as time variable $t$ goes to $+\infty$ is a
major problem. Most recently, there has been some progress along this line
(see, for example, \cite{2}, \cite{3} and \cite{9}).

For $\nabla u\in L^{2}(0,T;L^{2}\left(  \Omega\right)  )$, it known that for
$\left\Vert \nabla u_{0}\right\Vert _{L^{2}}^{2}+\dfrac{2}{\nu}\int_{0}%
^{T}\left\Vert f\right\Vert _{L^{2}}^{2}ds\leq c\lambda_{1}^{1/2}\nu^{2}$ we
have global regularity Constantin \cite[P. 80 Theorem 9.3]{4}. In the work of
Beirao 1995 \cite{1}, it has been proved that the regularity is ensured if
$\nabla u\in L^{4}(0,T;L^{2}\left(  \Omega\right)  )$.

In this paper, we prove global regularity of strong solutions to the 3D
Navier-Stokes problem in the classe $\nabla u\in L^{\alpha}\left(
0,T,L^{2}\left(  \Omega\right)  \right)  $ with $2\leq\alpha\leq4$. We give a
sufficient condition for global existence in time for strong solution to the
3D Navier-Stokes equations with external force and large data.

\section{The regularity criterion}

In this paper, we consider the three-dimensional Navier-Stokes system
\begin{equation}%
\begin{array}
[c]{c}%
\dfrac{\partial u}{\partial t}+u.\nabla u=\nu\triangle u+\nabla p+f,\text{
}t>0,\\
\text{div }u=0,\text{ in }\Omega\times\left(  0,\infty\right)  \text{,
}u=0\text{ on }\partial\Omega\times\left(  0,\infty\right)  \text{ and
}u\left(  x,0\right)  =u_{0},\text{ in }\Omega\text{,}%
\end{array}
\label{1}%
\end{equation}
where $u=u\left(  x,t\right)  \ $is the velocity vector field, $p\left(
x,t\right)  $ is a scalar pressure, $f$ is a given force field and $\nu>$ $0$
is the viscosity of the fluid. $u\left(  x,0\right)  $ with div$u_{0}=0$ in
the sense of distribution is the initial velocity field, and $\Omega$ is a
regular, open, bounded subset of $%
\mathbb{R}
^{3}$ with smooth boundary $\partial\Omega$. We denote by $H^{m}\left(
\Omega\right)  $, the Sobolev space. We define the usual function spaces
$\mathcal{V}=\left\{  u\in C_{0}^{\infty}\left(  \Omega\right)  \text{: div
}u=0\right\}  $, $V=$ closure of $\mathcal{V}$ in $H_{0}^{1}\left(
\Omega\right)  $, $H=$ closure of $\mathcal{V}$ in $L^{2}(\Omega)$. The space
$H$ is equipped with the scalar product $(.,.)$ induced by $L^{2}(\Omega)$ and
the norm $\left\Vert .\right\Vert _{L^{2}}=\left\Vert .\right\Vert
_{L^{2}\left(  \Omega\right)  }$. For the local existence of strong solutions,
we have the following result in 3D \cite{4}.

\begin{theorem}
\cite[P. 82 Theorem 9.4]{4} Assume that $u_{0}\in V$ and $f\in L^{2}(0,T;H)$
are given. Then there exists a $T^{\ast}>0$ depending on $\nu$, $f$, $u_{0}$
and $T$, such that there exists a unique solution of $(\ref{1})$ satisfies
$u\in L^{\infty}(0,T^{\ast};V)\cap L^{2}\left(  0,T^{\ast};H^{2}\left(
\Omega\right)  \right)  $.
\end{theorem}

For $f=0$, our main result on the global regularity for strong solutions with
large data to the 3D Navier-Stokes equations $(\ref{1})$ reads as follows

\begin{theorem}
Let $u_{0}\in V$ and $u\left(  x,t\right)  $ is the corresponding strong
solution to the system $(\ref{1})$ on $[0,T]$. Then if $u$ satisfies
\begin{equation}
\int_{0}^{T}\left\Vert \nabla u\left(  .,t\right)  \right\Vert _{L^{2}%
}^{\alpha}dt<C_{1}\left(  \nu,u_{0}\right)  \text{ for }2\leq\alpha
\leq4\label{2}%
\end{equation}
the strong solution $u\left(  x,t\right)  $ exists globally in time for all
finite value of $\left\Vert \nabla u_{0}\right\Vert _{L^{2}}$.
\end{theorem}

\begin{proof}
Multiplying $(\ref{1})$ by $\triangle u$ and integrate over $\Omega$, we
obtain
\begin{equation}
\frac{1}{2}\frac{d}{dt}\left\Vert \nabla u\left(  .,t\right)  \right\Vert
_{L^{2}}^{2}+\nu\Vert\triangle u\Vert_{L^{2}}^{2}=\left(  \left(  u.\nabla
u\right)  .\triangle u\right)  .\label{3}%
\end{equation}
We will denote by $c_{i\in%
\mathbb{N}
}$, positive constants. Using the Holder inequality and the Sobolev theorem,
we get
\begin{equation}
\left\vert \left(  \left(  u.\nabla u\right)  .\triangle u\right)  \right\vert
\leq c_{1}\left\Vert \nabla u\right\Vert _{L^{2}}^{\frac{3}{2}}\left\Vert
\triangle u\right\Vert _{L^{2}}^{\frac{3}{2}},\label{4}%
\end{equation}
see \cite[P. 79 (2.22)]{4}. Combining $(\ref{3})$ and $(\ref{4})$, we obtain%
\begin{equation}
\frac{d}{dt}\left\Vert \nabla u\left(  .,t\right)  \right\Vert _{L^{2}}%
^{2}+\nu\Vert\triangle u\Vert_{L^{2}}^{2}\leq c_{1}\left\Vert \nabla
u\right\Vert _{L^{2}}^{\frac{3}{2}}\left\Vert \triangle u\right\Vert _{L^{2}%
}^{\frac{3}{2}}.\label{8}%
\end{equation}
Setting $y=\left\Vert \nabla u\left(  .,t\right)  \right\Vert _{L^{2}}^{2}$ in
$(\ref{8})$, we get
\begin{equation}
y\prime\left(  t\right)  +\nu\Vert\triangle u\Vert_{L^{2}}^{2}\leq
c_{1}y^{\frac{3}{4}}\left\Vert \triangle u\right\Vert _{L^{2}}^{\frac{3}{2}%
}.\label{9}%
\end{equation}
Dividing $(\ref{9})$ by $\left(  y+1\right)  ^{2}$ we get%
\begin{equation}
\frac{1}{\left(  y+1\right)  ^{2}}y\prime\left(  t\right)  +\nu\frac
{\Vert\triangle u\Vert_{L^{2}}^{2}}{\left(  y+1\right)  ^{2}}\leq
c_{1}y^{\frac{3}{4}}\frac{\left\Vert \triangle u\right\Vert _{L^{2}}^{\frac
{3}{2}}}{\left(  y+1\right)  ^{2}}\text{,}\label{10}%
\end{equation}
this is equivalant to%
\begin{equation}
\frac{1}{\left(  y+1\right)  ^{2}}y\prime\left(  t\right)  +\frac{\nu}{\left(
y+1\right)  ^{2}}\left\Vert \triangle u\right\Vert _{L^{2}}^{2}\leq
c_{2}y^{\frac{\alpha}{8}}\frac{\left\Vert \triangle u\right\Vert _{L^{2}%
}^{\frac{3}{2}}}{\left(  y+1\right)  ^{\frac{3}{2}}}\frac{y^{\frac{6-\alpha
}{8}}}{\left(  y+1\right)  ^{\frac{1}{2}}}\text{ with }2\leq\alpha
\leq4.\label{11}%
\end{equation}
For $2\leq\alpha\leq4$ we have $\frac{y^{\frac{6-\alpha}{8}}}{\left(
y+1\right)  ^{\frac{1}{2}}}\leq1$, then yields
\begin{equation}
\frac{1}{\left(  y+1\right)  ^{2}}y\prime\left(  t\right)  +\frac{\nu}{\left(
y+1\right)  ^{2}}\left\Vert \triangle u\right\Vert _{L^{2}}^{2}\leq
c_{3}y^{\frac{\alpha}{8}}\frac{\left\Vert \triangle u\right\Vert _{L^{2}%
}^{\frac{3}{2}}}{\left(  y+1\right)  ^{\frac{3}{2}}}.\label{12}%
\end{equation}
Applying Young's inequality on the second member of $(\ref{12})$ with $p=4/3$
and $q=4$ we obtain
\begin{equation}
\frac{1}{\left(  y+1\right)  ^{2}}y\prime\left(  t\right)  +\nu\frac
{\left\Vert \triangle u\right\Vert _{L^{2}}^{2}}{\left(  y+1\right)  ^{2}}\leq
c_{4}\nu^{3}y^{\frac{\alpha}{2}}+\nu\frac{\left\Vert \triangle u\right\Vert
_{L^{2}}^{2}}{2\left(  y+1\right)  ^{2}},\label{13}%
\end{equation}
wich gives%
\begin{equation}
\frac{1}{\left(  y+1\right)  ^{2}}y\prime\left(  t\right)  +\nu\frac
{\left\Vert \triangle u\right\Vert _{L^{2}}^{2}}{2\left(  y+1\right)  ^{2}%
}\leq c_{5}\nu^{3}y^{\frac{\alpha}{2}}.\label{14}%
\end{equation}
If we drop the positive term $\nu\frac{\left\Vert \triangle u\right\Vert
_{L^{2}}^{2}}{2\left(  y+1\right)  ^{2}}$ in $(\ref{14})$ then we have%
\begin{equation}
\frac{1}{\left(  y+1\right)  ^{2}}y\prime\left(  t\right)  \leq c_{6}%
y^{\frac{\alpha}{2}}\text{ with }c_{6}=c_{5}\nu^{3}.\label{15}%
\end{equation}
Multiplying $(\ref{15})$ by $\sin\left(  \frac{1}{y+1}\right)  $ and using the
fact that $\sin\left(  \frac{1}{y+1}\right)  \leq1$, we obtain%
\begin{equation}
\sin\left(  \frac{1}{y+1}\right)  \frac{1}{\left(  y+1\right)  ^{2}}%
y\prime\left(  t\right)  \leq c_{6}y^{\frac{\alpha}{2}}.\label{16}%
\end{equation}
Integrating $(\ref{16})$ on $[0,t]$, $t\leq T$, we get
\begin{equation}
\cos\left(  \frac{1}{y+1}\right)  \leq\cos\left(  \frac{1}{y_{0}+1}\right)
+c_{6}\int_{0}^{T}y^{\frac{\alpha}{2}}dt.\label{17}%
\end{equation}
For all finite $y_{0}$ we have $\cos\left(  \frac{1}{y_{0}+1}\right)  <1$,
then we can choose
\begin{equation}
\int_{0}^{T}y^{\frac{\alpha}{2}}dt<\frac{1-\cos\left(  \frac{1}{y_{0}%
+1}\right)  }{2c_{6}}=C_{1}\left(  \nu,\left\Vert \nabla u_{0}\right\Vert
_{L^{2}}\right)  .\label{18}%
\end{equation}
Thus, $(\ref{18})$ yields $\cos\left(  \frac{1}{y+1}\right)  <1$ and therefore
$\frac{1}{y+1}$ can not identically equal to zero wich means that $y\left(
t\right)  $ is finite for all $t\geq0$ and consequently $\nabla u\left(
.,t\right)  \in L^{\infty}\left(  0,T;L^{2}\left(  \Omega\right)  \right)  $.
This proves, in particular that $(\ref{2})$ is a condition sufficient for
global regularity.
\end{proof}

For the 3D Navier-Stokes equations with no-negligible forces we prove the
following result

\begin{proposition}
Let $u_{0}\in V$ , $f\in L^{2}\left(  0,T;L^{2}\left(  \Omega\right)  \right)
$ and $u\left(  x,t\right)  $ is the corresponding strong solution to the
system $(\ref{1})$ on $[0,T]$. Then if $u$ satisfies
\begin{equation}
\int_{0}^{T}\left\Vert \nabla u\left(  .,t\right)  \right\Vert _{L^{2}%
}^{\alpha}dt<C_{2}\left(  \nu,u_{0},f\right)  \text{ for }2\leq\alpha
\leq4\label{19}%
\end{equation}
the strong solution $u\left(  x,t\right)  $ exists globally in time for all
finite value of $\left\Vert \nabla u_{0}\right\Vert _{L^{2}}$.
\end{proposition}

\begin{proof}
Using the Schwartz and Young inequalities, we obtain that
\begin{align}
\left\vert \left(  f,\triangle u\right)  \right\vert  &  \leq\left\Vert
f\right\Vert _{L^{2}}\left\Vert \triangle u\right\Vert _{L^{2}}\label{20}\\
&  \leq c_{7}\left\Vert f\right\Vert _{L^{2}}^{2}+\frac{\nu}{4}\left\Vert
\triangle u\right\Vert _{L^{2}}^{2}.\nonumber
\end{align}
Combining $(\ref{3})$, $(\ref{4})$ and $(\ref{20})$, we have
\begin{equation}
\frac{d}{dt}\left\Vert \nabla u\left(  .,t\right)  \right\Vert _{L^{2}}%
^{2}+\nu\Vert\triangle u\Vert_{L^{2}}^{2}\leq c_{7}\left\Vert f\right\Vert
_{L^{2}}^{2}+c_{1}y^{\frac{3}{4}}\left\Vert \triangle u\right\Vert _{L^{2}%
}^{\frac{3}{2}}.\label{21}%
\end{equation}
Dividing $(\ref{21})$ by $\left(  y+1\right)  ^{2}$ and since $\frac
{1}{\left(  y+1\right)  ^{2}}\leq1$, it follows that%
\begin{equation}
\frac{1}{\left(  y+1\right)  ^{2}}y\prime\left(  t\right)  +\nu\frac
{\Vert\triangle u\Vert_{L^{2}}^{2}}{\left(  y+1\right)  ^{2}}\leq
c_{7}\left\Vert f\right\Vert _{L^{2}}^{2}+c_{1}y^{\frac{3}{4}}\frac{\left\Vert
\triangle u\right\Vert _{L^{2}}^{\frac{3}{2}}}{\left(  y+1\right)  ^{2}%
}\text{.}\label{22}%
\end{equation}
The proof of $(\ref{19})$ follows in the same manner as that of $(\ref{2})$.
After some manipulations similaire to $(\ref{3})$-$(\ref{16})$ and integration
on $\left[  0,T\right]  $, we find%
\begin{equation}
\cos\left(  \frac{1}{y+1}\right)  \leq\cos\left(  \frac{1}{y_{0}+1}\right)
+c_{7}\int_{0}^{T}\left\Vert f\right\Vert _{L^{2}}^{2}+c_{6}\int_{0}%
^{T}y^{\frac{\alpha}{2}},\label{23}%
\end{equation}
which implies if
\begin{equation}
\cos\left(  \frac{1}{y_{0}+1}\right)  +c_{7}\int_{0}^{T}\left\Vert
f\right\Vert _{L^{2}}^{2}+c_{6}\int_{0}^{T}y^{\frac{\alpha}{2}}<1,\label{24}%
\end{equation}
then we have global regularity and $\nabla u\left(  .,t\right)  \in L^{\infty
}\left(  0,\infty;L^{2}\left(  \Omega\right)  \right)  $. As before, we can
choose
\begin{equation}
C_{2}\left(  \nu,u_{0},f\right)  =\frac{1}{2c_{6}}\left(  1-\cos\left(
\frac{1}{y_{0}+1}\right)  -c_{7}\int_{0}^{T}\left\Vert f\right\Vert _{L^{2}%
}^{2}\right)  .\label{25}%
\end{equation}
This finishes the proof of $(\ref{19})$.
\end{proof}

\end{document}